\documentclass{amsart}%
\usepackage{amsfonts}
\usepackage{amsmath}
\usepackage{amssymb}
\usepackage{graphicx}%
\providecommand{\U}[1]{\protect\rule{.1in}{.1in}}
\newtheorem{theorem}{Theorem}
\theoremstyle{plain}

\newtheorem{corollary}{Corollary}

\newtheorem{definition}{Definition}

\newtheorem{remark}{Remark}

\numberwithin{equation}{section}
\begin{document}
\title[ ]{Kantorovich's Mass Transport Problem for Capacities}
\author{Sorin G. Gal}
\address{Department of Mathematics and Computer Science\\
University of Oradea\\
University\ Street No. 1, Oradea, 410087, Romania}
\email{galso@uoradea.ro}
\author{Constantin P. Niculescu}
\address{Department of Mathematics, University of Craiova, Craiova 200585, Romania\\}
\email{cpniculescu@gmail.com}
\thanks{Corresponding author: Constantin P. Niculescu. E-mail: cpniculescu@gmail.com}
\thanks{Published in Proceedings of the Romanian Academy Series A, \textbf{20} (2019),
No. 4. 337-345. }
\date{December 14, 2019}

\begin{abstract}
The aim of the present paper is to extend Kantorovich's mass transport problem
to the framework of upper continuous capacities and to prove the cyclic
monotonicity of the supports of optimal solutions. As in the probabilistic
case, this easily yields the corresponding extension of the Kantorovich duality.

\end{abstract}
\maketitle

\section{Introduction}

Kantorovich's mass transport problem was formulated in 1942, when Kantorovich
\cite{K1942} published a note containing the following explicit description of
it in probabilistic terms. Suppose that $X$ and $Y$ are two compact metric
spaces and $c:X\times Y\rightarrow\lbrack0,\infty)$ is a Borel-measurable map
referred to as a \emph{cost} function. Given two Borel probability measures
$\mu$ and $\nu,$ defined respectively on the Borel $\sigma$-algebras
$\mathcal{B}(X)$ and $\mathcal{B}(Y),$ a Borel probability measure $\pi$ on
$\mathcal{B}(X\times Y)=\mathcal{B}(X)\otimes\mathcal{B}(Y)$ is called a
\emph{transport plan} for $\mu$ and $\nu$ if $\mu$ is the projection of $\pi$
on $X$ and $\nu$ is the projection of $\pi$ on $Y,$ that is,
\[
\pi(A\times Y)=\mu(A)\text{ for all }A\in\mathcal{B}(X)\text{ }%
\]
and%
\[
\pi(X\times B)=\nu(B)\text{ for all }B\in\mathcal{B}(Y).
\]
Under the above conditions one said that $\mu$ and $\nu$ are the
\emph{marginals} of $\pi.$ The set $\Pi(\mu,\nu)$, of all transport plans for
$\mu$ and $\nu,$ is always nonempty (it contains at least the product measure
$\mu\otimes\nu)$ and also convex and weak star compact as a subset of
$C(X\times Y)^{\ast}$ $($the dual of the Banach space $C(X\times Y)$ of all
continuous functions $f:X\times Y\rightarrow\mathbb{R})$. This is a
combination of the Banach-Alaoglu theorem and the Riesz representation
theorem. Based on this fact, Kantorovich has noticed that the functional
\begin{equation}
\operatorname*{Cost}(\pi)=\int_{X\times Y}c(x,y)d\pi(x,y),\text{\quad}\pi
\in\Pi(\mu,\nu), \label{f1}%
\end{equation}
attains its infimum in the case of continuous cost functions. The plans $\pi$
at which the infimum is attained are called \emph{optimal transport plans}.

Kantorovich's problem is a relaxation of the Monge problem on "excavation and
embankments" (see \cite{M}), which refers to the minimization of the
\emph{transference cost},%
\[
\int_{X}c(x,T(x))d\mu,
\]
over all Borel measurable mappings $T:X\rightarrow Y$ that \emph{push forward}
$\mu$ to $\nu$ (that is, $\nu(A)=\mu(T^{-1}(A))$ for every $A\in
\mathcal{B}(Y)).$ The transport plan associated to such a mapping $T$ is%
\[
\pi_{T}(B)=\mu\left(  \left\{  x:(x,T(x))\in B\right\}  \right)  \text{\quad
for every }B\in\mathcal{B}(X\times Y).
\]
As a consequence, any optimal transport plan can be viewed as a generalized
solution for Monge's problem.

Kantorovich \cite{K1948} realized this connection in 1948 and since then one
speaks on the Monge--Kantorovich problem, a fusion of the two problems into a
vast subject with deep applications in economics, dynamical systems,
probability and statistics, information theory etc. Details are covered by a
number of excellent surveys and fine books published by Ambrosio \cite{Amb},
Evans \cite{Evans}, Galichon \cite{Galichon}, Gangbo and McCann \cite{GM},
Rachev \cite{Ra}, Rachev and R\"{u}schendorf \cite{RR}, Santambrogio
\cite{San} and Villani \cite{Vil2003}, \cite{Vil2009}, just to cite a few.

The aim of the present paper is to extend Kantorovich's mass transport problem
to the framework of upper/lower continuous capacities and to prove the
cyclical monotonicity of the supports of optimal supermodular plans. As in the
probabilistic case, this easily yields the corresponding extension of the
Kantorovich duality.

The concept of capacity (a kind of monotone set function not necessarily
additive) and the integral associated to it were introduced by Choquet
\cite{Ch1954} \cite{Ch1986} in the early 1950s, motivated by some problems in
potential theory. Nowadays they also become powerful tools in decision making
under risk and uncertainty, game theory, ergodic theory, pattern recognition,
interpolation theory etc. See Adams \cite{Adams},\ Denneberg \cite{Denn},
F\"{o}llmer and Schied \cite{FS}, Wang and Klir \cite{WK} and Wang and Yan
\cite{WY}, as well as the references therein.

Our paper is, to the best of our knowledge, the first to investigate up to
what extent the mass transport theory extends to a context marked by
uncertainty and incomplete knowledge. The necessary background on capacities
and Choquet integral makes the subject of Section 2. The main result of
Section 3 is Theorem 2, which asserts the existence of optimal transport
plans. The critical ingredient in the proof is a nonlinear version of the
Riesz representation theory, due to Epstein and Wang (see Theorem 1). The fact
that the support of any supermodular optimal transport plan is necessarily a
$c$-cyclically monotone set is proved in Theorem 4, Section 4. Actually, for a
given pair of marginals, there is a $c$-cyclically monotone set including the
supports of all supermodular optimal transport plans. See Corollary 2. The
paper ends by noticing that this fact together with a previous result due to
R\"{u}schendorf \cite{Rusch1991c} and Smith and Knott \cite{SK}, easily yield
the extension of the Kantorovich duality in the framework of capacities.

\section{Preliminaries on capacities and Choquet integral}

For the convenience of the reader we will briefly recall some basic facts
about capacities and Choquet integral.

Let $(X,\mathcal{A})$\ be an arbitrarily fixed measurable space, consisting of
a nonempty abstract set $X$ and a $\sigma$-algebra ${\mathcal{A}}$ of subsets
of $X.$

\begin{definition}
\label{def1} A set function $\mu:{\mathcal{A}}\rightarrow\lbrack0,1]$ is
called a capacity if it verifies the following two conditions:

$(a)$ $\mu(\emptyset)=0$ and $\mu(X)=1;$

$(b)~\mu(A)\leq\mu(B)$ for all $A,B\in{\mathcal{A}}$, with $A\subset B$.
\end{definition}

In applications some additional properties are useful.

A capacity $\mu$ is called upper continuous\emph{\ (}or continuous by
descending sequences\emph{)} if
\[
\lim_{n\rightarrow\infty}\mu(A_{n})=\mu\left(
{\displaystyle\bigcap_{n=1}^{\infty}}
A_{n}\right)
\]
for every nonincreasing sequence $(A_{n})_{n}$ of sets in $\mathcal{A};$ $\mu$
is called lower continuous (or continuous by ascending sequences), if
$\lim_{n\rightarrow\infty}\mu(A_{n})=\mu(\cup_{n=1}^{\infty}A_{n})$ for every
nondecreasing sequence $(A_{n})_{n}$ of sets in $\mathcal{A}$.

The upper/lower continuity of a capacity is a generalization of countable
additivity of an additive measure. Indeed, if $\mu$ is an additive capacity,
then upper/lower continuity is the same with countable additivity.

A capacity $\mu$ is called \emph{supermodular} if%
\[
\mu(A)+\mu(B)\leq\mu(A\cup B)+\mu(A\cap B)\text{ for all }A,B\in\mathcal{A};
\]
$\mu$ is called \emph{submodular} if the last inequality works in the opposite direction.

A simple way to construct nontrivial examples of upper continuous supermodular
capacities is to start with a probability measure $P:\mathcal{A\rightarrow
}[0,1]$ and to consider any nondecreasing, convex and continuous function
$u:[0,1]\rightarrow\lbrack0,1]$ such that $u(0)=0$ and $u(1)=1;$ for example,
one may chose $u(t)=t^{a}$ with $\alpha>1.$ Then the \emph{distorted
probability} $\mu=u(P)$ is an upper continuous supermodular capacity on the
$\sigma$-algebra $\mathcal{A}$.

Suppose that $\left(  X,\mathcal{A}\right)  $ and $\left(  Y,\mathcal{B}%
\right)  $ are two measurable spaces. Any (upper continuous) capacity
$\mu:\mathcal{A}\rightarrow\lbrack0,1]$ and any measurable mapping
$T:X\rightarrow Y$ induce a (upper continuous) capacity $T\#\mu$ called the
\emph{push-forward} of $\mu$ through $T$ and defined by the formula%
\[
\left(  T\#\mu\right)  \left(  B\right)  =\mu\left(  T^{-1}\left(  B\right)
\right)  \text{ for all }B\in{\mathcal{B}}.
\]
The main feature of this kind of capacities is the following change of
variables formula%
\begin{equation}
(C)\int_{Y}g(y)d\left(  T\#\mu\right)  =(C)\int_{X}g(T(x))d\mu\label{cvar}%
\end{equation}
which works for all nonnegative bounded random variables $g:Y\rightarrow
\mathbb{R}.$

The concept of integrability with respect to a capacity refers to the whole
class of random variables, that is, to all functions $f:X\rightarrow
\mathbb{R}$ verifying the condition of ${\mathcal{A}}$-measurability
($f^{-1}(A)\in{\mathcal{A}}$ for every set $A\in\mathcal{B}(\mathbb{R}%
)$).\qquad\qquad\qquad\qquad

\begin{definition}
\label{def2}The Choquet integral of a random variable $f$ with respect to the
capacity $\mu$ is defined as the sum of two Riemann improper integrals,
\[
(C)\int_{X}fd\mu=\int_{0}^{+\infty}\mu\left(  \{x\in X:f(x)\geq t\}\right)
dt+\int_{-\infty}^{0}\left[  \mu\left(  \{x\in X:f(x)\geq t\}\right)
-1\right]  dt,
\]
Accordingly, $f$ is said to be Choquet integrable if both integrals above are finite.
\end{definition}

If $f\geq0$, then the last integral in the formula appearing in Definition
\ref{def2} is 0.

The inequality sign $\geq$ in the above two integrands can be replaced by $>;
$ see \cite{WK}, Theorem 11.1,\emph{\ }p. 226.

Every bounded random variable is Choquet integrable. The Choquet integral
coincides with the Lebesgue integral when the underlying set function $\mu$ is
a $\sigma$-additive measure.

The Choquet integral of a function $f:X\rightarrow\mathbb{R}$ over a set
$A\in\mathcal{A}$ is defined via the formula%
\[
(C)\int_{A}fd\mu=\mu(A)\cdot(C)\int_{X}fd\mu_{A},
\]
where $\mu_{A}$ is the capacity defined by $\mu_{A}(B)=\mu(B\cap A)/\mu(A)$
for all $B\in\mathcal{A}.$

We next summarize some basic properties of the Choquet integral.

\begin{remark}
\label{rem1}$(a)$ If $\mu:{\mathcal{A}}\rightarrow\lbrack0,1]$ is a capacity
and $A\in\mathcal{A}$, then the associated Choquet integral is a functional on
the space of all bounded random variables defined on A such that:
\begin{gather*}
f\geq0\text{ implies }(C)\int_{A}fd\mu\geq0\text{ \quad\emph{(}
positivity\emph{)}}\\
f\leq g\text{ implies }\left(  C\right)  \int_{A}fd\mu\leq\left(  C\right)
\int_{A}gd\mu\text{ \quad\emph{(}monotonicity\emph{)}}\\
\left(  C\right)  \int_{A}afd\mu=a\cdot\left(  \left(  C\right)  \int_{A}%
fd\mu\right)  \text{ for }a\geq0\text{ \quad\emph{(}positive\emph{\ }%
homogeneity\emph{)}}\\
\left(  C\right)  \int_{A}1\cdot d\mu=\mu(A)\text{\quad\emph{(}%
calibration\emph{)}};
\end{gather*}
see \emph{\cite{Denn}}, p. \emph{64}, Proposition \emph{5.1} $(ii),$ for the
proof of positive\emph{\ }homogeneity.

$(b)$ In general, the Choquet integral is not additive but, if the bounded
random variables $f$ and $g$ are comonotonic \emph{(}that is, $(f(\omega
)-f(\omega^{\prime}))\cdot(g(\omega)-g(\omega^{\prime}))\geq0,$ for all
$\omega,\omega^{\prime}\in A$\emph{), }then
\[
\left(  C\right)  \int_{A}(f+g)d\mu=\left(  C\right)  \int_{A}fd\mu+\left(
C\right)  \int_{A}gd\mu.
\]
This is usually referred to as the property of comonotonic additivity. An
immediate consequence is the property of translation invariance,
\[
\left(  C\right)  \int_{A}(f+c)d\mu=\left(  C\right)  \int_{A}fd\mu+c\cdot
\mu(A)
\]
for all $c\in\mathbb{R}$ and all \hspace{0in}bounded random variables $f.$ See
\emph{\cite{Denn}}, Proposition \emph{5.1, }$(vi)$, p. \emph{65.}

$(c)$ If $\mu$ is an upper continuous capacity, then the Choquet integral is
upper continuous in the sense that
\[
\lim_{n\rightarrow\infty}\left(  \left(  C\right)  \int_{A}f_{n}d\mu\right)
=\left(  C\right)  \int_{A}fd\mu,
\]
whenever $(f_{n})_{n}$ is a nonincreasing sequence of bounded random variables
that converges pointwise to the bounded variable $f.$ This is a consequence of
the Bepo Levi monotone convergence theorem from the theory of Lebesgue
integral\emph{\ (}see\emph{\ \cite{Din}, Theorem 2, p. 133).}

$(d)$ If $\mu$ is a supermodular capacity, then\ the associated Choquet
integral is a superadditive functional, that is
\[
\left(  C\right)  \int_{A}fd\mu+\left(  C\right)  \int_{A}gd\mu\leq\left(
C\right)  \int_{A}(f+g)d\mu
\]
for all bounded random variables $f$ and $g.$ It is also a supermodular
functional in the sense that
\[
\left(  C\right)  \int_{A}\sup\left\{  f,g\right\}  d\mu+\left(  C\right)
\int_{A}\inf\{f,g\}d\mu\geq\left(  C\right)  \int_{A}fd\mu+(C)\int_{A}gd\mu
\]
for all bounded random variables $f$ and $g.$
\end{remark}

In this paper we are interested in a special kind of measurable spaces, those
of the form $(X,\mathcal{B}(X)),$ where $X$ is a compact metric space and
$\mathcal{B}(X)$ is the $\sigma$-algebra of all Borel subsets of $X.$ We will
denote by $\operatorname*{Ch}(X)$ the class of all upper continuous capacities
$\mu:\mathcal{B}(X)\rightarrow\lbrack0,1]$.

The following analogue of the Riesz representation theorem is due to L. G.
Epstein and T. Wang. See \cite{EW1996}, Theorem 4.2. See also Zhou
\cite{Zhou}, Theorem 1 and Lemma 3, for a simple (and more general) argument.

\begin{theorem}
\label{thm1}Suppose that $I:C(X)\rightarrow\mathbb{R}$ is a comonotonically
additive and monotone functional and $I(1)=1$. Then it is also upper
continuous and there exists a unique upper continuous capacity $\mu
:\mathcal{B}(X)\rightarrow\lbrack0,1]$ such that $I$ coincides with the
Choquet integral associated to it.

On the other hand, according to Remark \ref{rem1}, the Choquet integral
associated to any upper continuous capacity is a comonotonically additive,
monotone and upper continuous functional.
\end{theorem}

This result allows us to identify $\operatorname*{Ch}(X)$ with the set of all
functionals $I$ on $C(X)$ which are comonotonically additive, monotone and
verify $I(1)=1$. An immediate consequence is as follows:

\begin{corollary}
\label{cor1}\emph{(O'Brien, W. Vervaat \cite{OBrien})} $\operatorname*{Ch}(X)$
is compact and metrizable with respect to the weak topology on
$\operatorname*{Ch}(X) $ induced by the duality
\[
\langle\cdot,\cdot\rangle:C(X)\times\operatorname*{Ch}(X)\rightarrow
\mathbb{R},\quad\langle f,\mu\rangle=(C)\int_{X}fd\mu.
\]

\end{corollary}

A direct argument for Corollary 1 makes the objective of Theorem 2 in
\cite{Zhou}, were it is noticed that the weak convergence on
$\operatorname*{Ch}(X)$ is equivalent with the convergence with respect to the
metric
\[
d_{\operatorname*{Ch}(X)}(\mu,\nu)=\sum_{j=1}^{\infty}\frac{1}{2^{j}%
}\left\vert (C)\int_{X}f_{j}d\mu-(C)\int_{X}f_{j}\,d\nu\right\vert ,
\]
associated to an arbitrary sequence $(f_{j})_{j}$ dense in the unit sphere of
$C(X)$.

\begin{remark}
\label{rem2}Under the assumptions of Theorem \emph{1}, the functional $I$ is
supermodular if and only if the capacity $\mu$ is supermodular. For details
see \emph{\cite{CMMM}}, Theorem \emph{13~(c)}. Notice also that the subset
$\operatorname*{Ch}^{\ast}(X)$ of all supermodular capacities $\mu
\in\operatorname*{Ch}(X)$ is closed with respect to the weak topology
\emph{(}and thus it is compact, according to Corollary \emph{\ref{cor1})}.
\end{remark}

\section{The existence of optimal transport plan}

The framework used in this section parallels that of Borel probability
measures, but the details are based on the integral representation of
comonotonically additive and monotone functionals provided by Theorem 1.

Let $X$ and $Y$ be two compact metric spaces on which there are given the
upper continuous capacities $\mu\in\operatorname*{Ch}(X)$ and respectively
$\nu\in\operatorname*{Ch}(Y).$ A \emph{transport plan} for $\mu$ and $\nu$ is
any capacity $\pi\in\operatorname*{Ch}(X\times Y)$ with marginals
$\mu=\operatorname*{pr}_{X}\#\pi$ and $\nu=\operatorname*{pr}_{Y}\#\pi,$ that
is, such that%
\[
\pi(A\times Y)=\mu(A)\text{ for all }A\in\mathcal{B}(X)\text{ }%
\]
and%
\[
\pi(X\times B)=\nu(B)\text{ for all }B\in\mathcal{B}(Y).
\]
The set of all such transport plans will be denoted $\Pi_{\operatorname*{Ch}%
}(\mu,\nu).$ $\Pi_{\operatorname*{Ch}}(\mu,\nu)$ is a nonempty set. For
example, if $P$ and $Q$ are probability measures, then $(P\otimes Q)^{\alpha}$
is a transport plan for the distorted probabilities $\mu=P^{\alpha}$ and
$\nu=Q^{\alpha},$ whenever $\alpha\geq1$.

Given a Borel measurable cost function $c:X\times Y\rightarrow\lbrack
0,\infty),$ the cost of a transport plan $\pi\in\Pi_{\operatorname*{Ch}}%
(\mu,\nu)$ is defined by a formula similar to formula (\ref{f1}):%
\[
\operatorname*{Cost}(\pi)=(C)\int_{X\times Y}c(x,y)d\pi(x,y).
\]
The existence of the optimal transport plans is motivated by the following result:

\begin{theorem}
\label{thm2}If the cost function $c:X\times Y\rightarrow\lbrack0,\infty)$ is
continuous, then there exists $\pi_{0}\in\Pi_{\operatorname*{Ch}}(\mu,\nu)$
such that
\[
\operatorname*{Cost}(\pi_{0})=\inf_{\pi\in\Pi_{\operatorname*{Ch}}(\mu,\nu
)}\operatorname*{Cost}(\pi).
\]

\end{theorem}

We will denote by $\pi\rightarrow$ $I_{\pi}$ the bijection stated by Theorem
\ref{thm1}, that makes possible to identify the set $\operatorname*{Ch}%
(X\times Y),$ of all upper continuous capacities on $X\times Y,$ with the set
$\Phi,$ of all comonotonically additive and monotone functionals $I$ on
$C(X\times Y)$ that verify $I(1)=1$. By this bijection, the set $\Pi
_{\operatorname*{Ch}}(\mu,\nu)$ corresponds to the subset $\Phi(\mu,\nu)$ of
$\Phi$, consisting of those functionals $I:C(X\times Y)\rightarrow\mathbb{R}$
such that%
\begin{equation}
I(\overline{u})=I_{\mu}(u)\text{ for all }u\in C(X) \label{m1}%
\end{equation}
and
\begin{equation}
I(\overline{w})=I_{\nu}(w)\text{ for all }w\in C(Y), \label{m2}%
\end{equation}
where $\overline{u}$ and respectively $\overline{w}$ represent the extensions
of $u$ and $w$ to $X\times Y$ via the formulas
\begin{equation}
\overline{u}(x,y)=u(x)\text{ and}~\overline{w}(x,y)=w(y)\text{ for all
}\left(  x,y\right)  \in X\times Y, \label{m3}%
\end{equation}
\noindent and $I_{\mu}:C(X)\rightarrow\mathbb{R}$ and $I_{\nu}:C(Y)\rightarrow
\mathbb{R}$ are the unique comonotonically additive and monotone functionals
generated by $\mu$ and $\nu$ via the Choquet integral. Indeed, assuming that
$I=I_{\pi},$ since $\mu$ and $\nu$ are the marginals of $\pi$, we have%
\[
\pi(\{(x,y)\in X\times Y:\overline{u}(x,y)\geq\alpha\})=\mu(\{x\in
X:u(x)\geq\alpha\}),
\]
and%
\[
\pi(\{(x,y)\in X\times Y:\overline{w}(x,y)\geq\alpha\})=\nu(\{y\in
Y:w(y)\geq\alpha\}),
\]
for all $\alpha\in\mathbb{R},$ whence, by the definition of the Choquet
integral, we infer that
\[
I_{\pi}(\overline{u})=I_{\mu}(u)\text{ and }I_{\pi}(\overline{w})=I_{\nu
}(w)\text{.}%
\]
We will use the above remark to prove that $\Phi(\mu,\nu)$ is a closed subset
of $\Phi$ (and thus compact, according to Corollary 1). Since the topology of
$\Phi$ is metrizable, this reduces to the fact that $\Phi(\mu,\nu)$ is closed
under the operation of taking countable limits. For this, let $(\gamma
_{n})_{n}$ be a sequence of elements of $\Phi(\mu,\nu)$ converging to a
capacity $\gamma\in\Phi$ and put $I_{n}=I_{\gamma_{n}}$ and $I=I_{\gamma}$ in
order to simplify the notation. By our assumptions, the functionals $I_{n}$
and $I$ are comonotonically additive and monotone and%
\begin{equation}
I_{n}(f)\rightarrow I(f)\text{\quad for all }f\in C(X\times Y). \label{cv1}%
\end{equation}
The membership of $I_{n}$ to $\Phi(\mu,\nu)$ translates into the formulas%
\begin{equation}
I_{n}(\overline{u})=I_{\mu}(u)\text{\quad for all }u\in C(X)
\end{equation}
and
\begin{equation}
I_{n}(\overline{w})=I_{\nu}(w)\text{\quad for all }w\in C(Y),
\end{equation}
where $\overline{u}$ and respectively $\overline{w}$ represent the extensions
of $u$ and $w$ to $X\times Y$ via the formulas (\ref{m3}). Combining this fact
with (\ref{cv1}), one easily conclude that $I$ verifies similar formulas and
thus the capacity $\gamma$ that generates $I$ belongs to $\Phi(\mu,\nu).$ In
order to end the proof let's choose a sequence $(\pi_{n})_{n}$ of elements of
$\Pi_{\operatorname*{Ch}}(\mu,\nu)$ which minimizes the cost function, that
is, such that
\[
\lim_{n\rightarrow\infty}\,\operatorname*{Cost}(\pi_{n})=m=\inf_{\gamma\in
\Pi_{\operatorname*{Ch}}(\mu,\nu)}\operatorname*{Cost}(\gamma).
\]
Since $\Pi_{\operatorname*{Ch}}(\mu,\nu)$ is a compact set, we may assume (by
passing to a subsequence if necessary) that $(\pi_{n})_{n}$ converges to some
$\pi\in\Pi_{\operatorname*{Ch}}(\mu,\nu).$ Then $\operatorname*{Cost}(\pi)=m,$
which means that $\pi$ is an optimal transport plan.

\begin{remark}
It is a simple exercise to prove that if $X$ is a compact metric space, then
every functional $I$ on $C(X)$ which is comonotonically additive and monotone
is also lower continuous. Combining this fact with Proposition \emph{17
(i)\&(ii)} and Corollary \emph{18} in \emph{\cite{CMMM}}, one can easily
obtain the validity of the "lower" version of Theorem $1$.
\end{remark}

As concerns the case of supermodular capacities, let us notice first that the
existence of a supermodular transport plan $\pi\in\Pi_{\operatorname*{Ch}}%
(\mu,\nu)$ imposes that both $\mu$ and $\nu$ are supermodular. See \cite{Ghi},
Lemma 4, p. 281. An inspection of the argument of Theorem 2 easily yields the
following result:

\begin{theorem}
\label{thm3}If the cost function $c:X\times Y\rightarrow\lbrack0,\infty)$ is
continuous and $\mu$ and $\nu$ are two supermodular upper continuous
capacities such that
\[
\Pi_{\operatorname*{Ch}^{\ast}}(\mu,\nu)=\left\{  \pi\in\Pi
_{\operatorname*{Ch}}(\mu,\nu):\pi\text{ supermodular}\right\}
\]
is nonempty, then there exists $\pi_{0}\in\Pi_{\operatorname*{Ch}^{\ast}}
(\mu,\nu) $ such that
\[
\operatorname*{Cost}(\pi_{0})=\inf_{\pi\in\Pi_{\operatorname*{Ch}^{\ast}}
(\mu,\nu)}\operatorname*{Cost}(\pi).
\]

\end{theorem}

\section{A necessary condition for the optimality of a transport plan}

The aim of this section is to prove that optimal transport plans have
$c$-cyclically monotone supports, a fact that relates them to the theory of
$c$-concave functions. For this, we need some preparation.

As above, $X$ and $Y$ are compact metric spaces and $c:X\times Y\rightarrow
\lbrack0,\infty)$ is a cost function.

\begin{definition}
\label{def3}A subset $S\subset X\times Y$ is called $c$-cyclically monotone if
for every finite number of points $(x_{i},y_{i})\in S,$ $i=1,...,n$, and any
permutation $\sigma$ of $\left\{  1,...,n\right\}  ,$ we have%
\[
\sum\limits_{i=1}^{n}c(x_{i},y_{i})\leq\sum\limits_{i=1}^{n}c(x_{\sigma
(i)},y_{i}).
\]

\end{definition}

By definition, the support of an upper continuous capacity $\gamma
:{\mathcal{B}}(X\times Y)\rightarrow\lbrack0,1]$ is the set
$\operatorname*{supp}(\gamma)$ of all points $(x,y)$ in $X\times Y$ for which
every open neighborhood $\mathcal{W}$ verifies $\gamma\left(  \mathcal{W}%
\right)  >0.$ The support is a closed set since its complement is the union of
the open sets of capacity zero.

We are now in a position to prove the following result.

\begin{theorem}
\label{thm4}If the cost function $c:X\times Y\rightarrow\lbrack0,\infty)$ is
continuous, then every supermodular optimal transport plan $\pi\in
\Pi_{\operatorname*{Ch}}(\mu,\nu)$ has a $c$-cyclically monotone support.
\end{theorem}

\begin{proof}
Our argument is close to that used by Ambrosio \cite{Amb}, Theorem 2.2, in the
probabilistic framework.
If $\operatorname*{supp}(\pi)$ were not $c$-cyclically monotone, then there
would exist points $(\bar{x}_{1},\bar{y}_{1}),...,\allowbreak(\bar{x}_{n}%
,\bar{y}_{n})$ in $\operatorname*{supp}(\pi)$ and a permutation $\sigma$ of
$\left\{  1,...,n\right\}  $ such that%
\[
\sum\limits_{i=1}^{n}c(\bar{x}_{i},\bar{y}_{i})>\sum\limits_{i=1}^{n}c(\bar
{x}_{\sigma(i)},\bar{y}_{i}).
\]
Choose $\varepsilon$ such that
\[
0<\varepsilon<\frac{1}{2n}\left(  \sum\limits_{i=1}^{n}c(\bar{x}_{i},\bar
{y}_{i})-\sum\limits_{i=1}^{n}c(\bar{x}_{\sigma(i)},\bar{y}_{i})\right)  .
\]
Since $c$ is a continuous function, there exists compact neighborhoods $U_{i}$
of $\bar{x}_{i}$ and $V_{i}$ of $\bar{y}_{i}$ such that $c(x_{i},y_{i}%
)>c(\bar{x}_{i},\bar{y}_{i})-\varepsilon$ for all $(x_{i},y_{i})\in
U_{i}\times V_{i}$ and $c(x_{i},y_{i})<c(\bar{x}_{\sigma(i)},\bar{y}%
_{i})+\varepsilon$ for all $(x_{i},y_{i})\in U_{\sigma(i)}\times V_{i}.$
We have $\pi\left(  U_{i}\times V_{i}\right)  >0,$ due to the fact that
$(x_{i},y_{i})\in\operatorname*{supp}(\pi).$ Then $\alpha=(1/n)\min_{i}%
\pi\left(  U_{i}\times V_{i}\right)  >0$ and we can consider the upper
continuous capacities $\pi_{i}$ defined by
\[
\pi_{i}(\mathcal{W})=\frac{\pi(\mathcal{W}\cap\left(  U_{i}\times
V_{i}\right)  )}{\pi\left(  U_{i}\times V_{i}\right)  }\text{\quad for every
}\mathcal{W}\in{\mathcal{B}}(X\times Y).
\]
The marginals of $\pi_{i}$ are $\mu_{i}=\operatorname*{pr}_{X}\#\pi_{i}$ and
$\mu_{i}=\operatorname*{pr}_{Y}\#\pi_{i},$ where $\operatorname*{pr}_{X}$ and
$\operatorname*{pr}_{Y}$ are the canonical projections of $X\times Y$
respectively on $X$ and $Y.$ The set function%
\[
\gamma=\pi-\alpha\sum_{i=1}^{n}\pi_{i}+\alpha\sum_{i=1}^{n}\mu_{i}\otimes
\nu_{i}%
\]
is nonnegative since $\pi-\alpha\sum_{i-1}^{n}\pi_{i}\geq0,$ according to the
choice of $\alpha.$ Moreover, $\gamma(\emptyset)=0$ and $\gamma(X\times Y)=1.$
In order to prove the monotonicity\ of $\gamma$, let's consider two Borel
subsets $\mathcal{V}$ and $\mathcal{W}$ of $X\times Y$ such that
$\mathcal{V}\subset\mathcal{W}.$
The inequality $\gamma(\mathcal{V})\leq\gamma(\mathcal{W})$ is a consequence
of the fact that%
\[
\pi(\mathcal{W})-\pi(\mathcal{V})\geq\alpha\sum_{i=1}^{n}\frac{\pi
(\mathcal{W}\cap\left(  U_{i}\times V_{i}\right)  )-\pi(\mathcal{V}\cap\left(
U_{i}\times V_{i}\right)  )}{\pi\left(  U_{i}\times V_{i}\right)  }.
\]
For this it suffices to show that
\[
\pi(\mathcal{W})-\pi(\mathcal{V})\geq\pi(\mathcal{W}\cap\left(  U_{i}\times
V_{i}\right)  )-\pi(\mathcal{V}\cap\left(  U_{i}\times V_{i}\right)  ),
\]
for all indices $i.$ We claim that actually%
\[
\pi(\mathcal{W})+\pi(\mathcal{V}\cap\mathcal{K})\geq\pi(\mathcal{V}%
)+\pi(\mathcal{W}\cap\mathcal{K})
\]
for every Borel subset $\mathcal{K}\subset X\times Y.$ Indeed,
\[
\chi_{\mathcal{W}}+\chi_{\mathcal{V}\cap\mathcal{K}}\geq\chi_{\mathcal{V}%
}+\chi_{\mathcal{W}\cap\mathcal{K}}%
\]
and because $\chi_{\mathcal{W}}$ and $\chi_{\mathcal{V}\cap\mathcal{K}}$ are
comonotonic and $\pi$ is monotone and supermodular we have%
\begin{align*}
\pi(\mathcal{W})+\pi(\mathcal{V}\cap\mathcal{K})  &  =(C)\int_{X\times Y}%
\chi_{\mathcal{W}}d\pi+(C)\int_{X\times Y}\chi_{\mathcal{V}\cap\mathcal{K}%
}d\pi\\
&  =(C)\int_{X\times Y}\left(  \chi_{\mathcal{W}}+\chi_{\mathcal{V}%
\cap\mathcal{K}}\right)  d\pi\geq(C)\int_{X\times Y}\left(  \chi_{\mathcal{V}%
}+\chi_{\mathcal{W}\cap\mathcal{K}}\right)  d\pi\\
&  =(C)\int_{X\times Y}\chi_{\mathcal{V}}d\pi+(C)\int_{X\times Y}%
\chi_{\mathcal{W}\cap\mathcal{K}}d\pi\\
&  =\pi(\mathcal{V})+\pi(\mathcal{W}\cap\mathcal{K}).
\end{align*}
This ends the proof of the monotonicity of $\gamma.$
Clearly, $\gamma$ is upper continuous and its marginals are $\mu$ and $\nu.$
The cost of the transport plan $\gamma$ is less than the cost of $\pi$
because
\begin{align*}
&  \operatorname*{Cost}(\pi)-\operatorname*{Cost}(\gamma)\\
&  =\alpha\sum_{i=1}^{n}\left(  (C)\int_{X\times Y}c(x,y)d\pi_{i}\right)
-\alpha\sum_{i=1}^{n}\left(  (C)\int_{X\times Y}c(x,y)d\mu_{i}\otimes\nu
_{i}\right) \\
&  \geq\alpha\sum_{i=1}^{n}\left(  c(\bar{x}_{i},\bar{y}_{i})-\varepsilon
\right)  -\alpha\sum\limits_{i=1}^{n}\left(  c(\bar{x}_{\sigma(i)},\bar{y}%
_{i})+\varepsilon\right) \\
&  =\alpha\sum_{i=1}^{n}\left(  c(\bar{x}_{i},\bar{y}_{i})-\sum\limits_{i=1}%
^{n}c(\bar{x}_{\sigma(i)},\bar{y}_{i})-2n\varepsilon\right)  >0.
\end{align*}
Since this contradicts the optimality of $\pi,$ we conclude that
$\operatorname*{supp}(\pi)$ is $c$-cyclically monotone.
\end{proof}

\begin{corollary}
\label{cor2}Under the assumption of Theorem $\ref{thm3}$, there is a
$c$-cyclically monotone subset of $X\times Y$ containing the supports of all
supermodular optimal transport plans.
\end{corollary}

\begin{proof}
Indeed, $\Pi_{\operatorname*{Ch}^{\ast}}(\mu,\nu)$ is a convex set, a fact
which easily implies that the union of all supports $\operatorname*{supp}%
(\pi)$ with $\pi\in\Pi_{\operatorname*{Ch}}^{\ast}(\mu,\nu)$ is a
$c$-cyclically monotone set.
\end{proof}

As was noticed by R\"{u}schendorf \cite{Rusch1991c} and Smith and Knott
\cite{SK}, the notion of $c$-cyclically monotone set is intimately related to
the theory of $c$-concave functions (concavity relative to a cost function).
One important fact in this connection is the existence for each $c$-cyclically
monotone subset $S$ of $X\times Y$ of a pair of continuous functions
$\varphi:X\rightarrow\mathbb{R}$ and $\psi:Y\rightarrow\mathbb{R}$ such that:

\smallskip

$(CM1)$ $\varphi(x)=\inf\left\{  c(x,y)-\psi(y):y\in Y\right\}  $ for all $x;$

$(CM2)$ $\psi(y)=\inf\left\{  c(x,y)-\varphi(x):x\in X\right\}  $ for all $y;$

$(CM3)$ $S\subset\left\{  \left(  x,y\right)  \in X\times Y:\varphi
(x)+\psi(y)=c(x,y)\right\}  .$

\smallskip

Combining this remark with Theorem \ref{thm3} one can easily deduce the
following extension of Kantorovich duality to the framework of supermodular
upper continuous capacities:%
\begin{multline*}
\min_{\gamma\in\Pi_{\operatorname*{Ch}^{\ast}}(\mu,\nu)}\operatorname*{Cost}%
(\gamma)=\\
\sup\left\{  (C)\int_{X}\varphi d\mu+(C)\int_{Y}\psi d\nu:(\varphi,\psi)\in
C(X)\times C(Y),\text{ }\varphi(x)+\psi(y)\leq c(x,y)\right\}  .
\end{multline*}

\begin{remark}
Theorem \emph{3} and Theorem \emph{4} also work in the context of lower
continuous supermodular capacities.
\end{remark}

\end{document}